\documentclass[12pt,a4paper]{amsart}
\usepackage{amssymb,latexsym}
\usepackage{enumerate}
\usepackage[T1]{fontenc}
\usepackage{lmodern}

\newtheorem{thm}{Theorem}[section]

\newtheorem{lem}[thm]{Lemma}
\newtheorem{pro}[thm]{Proposition}

\theoremstyle{definition}
\newtheorem{rem}[thm]{Remark}
\makeatletter
\renewcommand{\endrem}{\qed\endtrivlist \@endpefalse}
\makeatother

\numberwithin{equation}{section}
\newcommand{\R}{\mathbb{R}}
\newcommand{\Z}{\mathbb{Z}}
\newcommand{\N}{\mathbb{N}}
\newcommand{\diam}{\operatorname{diam}}
\newcommand{\abs}[1]{\lvert#1\rvert}
\newcommand{\norm}[1]{\lVert#1\rVert}
\newcommand{\floor}[1]{\lfloor#1\rfloor}
\newcommand{\babs}[1]{[#1]}
\newcommand{\cabs}[1]{\lceil#1\rceil}
\newcommand{\Gsn}{G_{\sigma,n}}
\newcommand{\Step}[1]{\medskip\noindent\textsc{Step~#1}:}

\begin{document}

\title[Limiting behaviour of Sobolev intrinsic semi-norms]{%
   Limiting behaviour of intrinsic semi-norms \\ in fractional order 
   Sobolev spaces \\[2ex] \small
   Dedicated to the memory of Charles Goulaouic (1938--1983)}
   
\author[R. Arcang\'eli]{R\'emi Arcang\'eli}
\address{Route de Barat, 31160 Arbas, France}
\email{arcangeli.remi@wanadoo.fr}

\author[J. J. Torrens]{Juan Jos\'e Torrens}
\address{Departamento de Ingenier\'ia Matem\'atica e Inform\'atica, 
         Universidad P\'u\-bli\-ca de Navarra,
	 Campus de Arro\-sa\-d\'ia, 31006 Pamplona, Spain} 
\email{jjtorrens@unavarra.es}

\date{\today}

\begin{abstract}
   We collect and extend results on the limit of
   $\sigma^{1-k}(1-\sigma)^{k}\abs{v}_{l+\sigma,p,\Omega}^p$ as
   $\sigma\to0^{+}$ or $\sigma\to1^{-}$, where $\Omega$ is $\R^n$ or a
   smooth bounded domain, $k\in\{0,1\}$, $l\in\N$,
   $p\in[1,\infty)$, and $\abs{\,\cdot\,}_{l+\sigma,p,\Omega}$ is the
   intrinsic semi-norm of order $l+\sigma$ in the Sobolev space
   $W^{l+\sigma,p}(\Omega)$. In general, the above limit is equal to
   $c[v]^p$, where $c$ and $[\,\cdot\,]$ are, respectively, a constant
   and a semi-norm that we explicitly provide. The particular case
   $p=2$ for $\Omega=\R^n$ is also examined and the results are then
   proved by using the Fourier transform.
\end{abstract}



\maketitle

% ======================================================================
% ======================================================================
% 1. Introduction
% ======================================================================
% ======================================================================

\section{Introduction}

Bourgain, Br\'ezis and Mironescu (cf. \cite{BouBreMir01,Bre02}) proved
that, for any $p\in[1,\infty)$ and any $v$ belonging to the Sobolev
space $W^{1,p}(\Omega)$,
\begin{equation}\label{eq:BBM}
   \lim_{\sigma\to1^{-}}(1-\sigma)\abs{v}_{\sigma,p,\Omega}^p
   =p^{-1} K_{p,n} \int_\Omega \abs{\nabla v(x)}^p\,dx,
\end{equation}
where $\Omega$ is either $\R^n$ or a smooth bounded domain in $\R^n$,
with $n\geq 1$, $\abs{\,\cdot\,}_{\sigma,p,\Omega}$ is the intrinsic
or Gagliardo semi-norm of order $\sigma$ in the Sobolev space
$W^{\sigma,p}(\Omega)$ (see Section~\ref{s2} for the precise
definitions), and $K_{p,n}$ is a constant that only depends on $p$ and
$n$. Likewise, Maz'ya and Shaposhnikova~\cite{MazSha02} showed that
\begin{equation}\label{eq:MS}
   \lim_{\sigma\to0^{+}}\sigma\abs{v}_{\sigma,p,\R^n}^p
   =2p^{-1}\abs{S_{n-1}} \abs{v}_{0,p,\R^n}^p,
\end{equation}
where $S_{n-1}$ stands for the unit sphere in $\R^n$ (i.e.
$S_{n-1}=\{x\in\R^n\mid\abs{x}=1\}$) and $\abs{S_{n-1}}$ is its area.

These results have been extended and completed by several authors. Let
us quote, for example, Milman~\cite{Mil05}, who placed them in the
frame of interpolation spaces, or Karadzhov, Milman and
Xiao~\cite{KarMilXia05}, Kolyada and Lerner~\cite{KolLer05} and
Triebel~\cite{Tri11}, who generalized them in the context of Besov
spaces.

Our interest in this subject comes from the study of sampling
inequalities involving Sobolev semi-norms. In~\cite{ArcLopTor10}, we
have extended previous results (cf.~\cite{ArcLopTor07,ArcLopTor09}) in
order to allow fractional order Sobolev semi-norms on the left-hand
side of sampling inequalities. We have then realized that the complete
comprehension of the constants involved in sampling inequalities needs
an understanding of the asymptotic behaviour of the corresponding
fractional order Sobolev semi-norms. In fact, we need extensions of
\eqref{eq:BBM} and \eqref{eq:MS} having the following form:
\begin{equation}\label{eq:intro0}
   \lim_{\sigma\to\ell} \sigma^{1-k}(1-\sigma)^{k}
   \abs{v}_{l+\sigma,p,\Omega}^p=c[v]^p,
\end{equation}
where $\ell=0^{+}$ or $1^{-}$, $\Omega$ is $\R^n$ or a smooth bounded
domain, $k\in\{0,1\}$, $l\in\N$, $p\in[1,\infty)$, and
$\abs{\,\cdot\,}_{l+\sigma,p,\Omega}$ is the intrinsic semi-norm of
order $l+\sigma$ in the Sobolev space $W^{l+\sigma,p}(\Omega)$. On the
right-hand side of \eqref{eq:intro0}, the notations $[\,\cdot\,]$
and~$c$ stand, respectively, for a semi-norm and a constant to be
specified.

The first part of this paper will be devoted to establish
\eqref{eq:intro0}. Most of the work may be routine, but anyway we find
it useful to collect and state in one place this kind of results and
to provide explicit expressions of the constants and semi-norms
involved in the limits.

In the second part of the paper, we shall focus on the case $p=2$ and
$\Omega=\R^n$. We show that \eqref{eq:intro0} can be obtained by means
of the Fourier transform. This line of reasoning was suggested in
\cite[Remark~2]{BouBreMir01} starting from a result by Masja and Nagel
\cite{MasNag78}. As a by-product, for $m\in\N$ and $s\geq0$, we
establish a relationship between the Sobolev space $W^{m+s,2}(\R^n)$
and the Beppo-Levi space $X^{m,s}$, which is a space that arises in
spline theory (cf.~\cite[Chapter~I]{ArcLopTor04}).

% ======================================================================
% ======================================================================
% 2. Preliminaries
% ======================================================================
% ======================================================================

\section{Preliminaries}\label{s2}

For any $x\in\R$, we shall write $\floor{x}$ for the \emph{floor} (or
integer part) of $x$, that is, the unique integer satisfying
$\floor{x}\leq x<\floor{x}+1$. The letter $n$ will always stand for an
integer belonging to $\N^\ast=\N\setminus\{0\}$ (by convention,
$0\in\N$). The Euclidean norm in $\R^n$ will be denoted by
$\abs{\,\cdot\,}$.

For any multi-index $\alpha=(\alpha_1,\ldots,\alpha_n)\in\N^n$, we
write $\abs{\alpha}=\alpha_1+\cdots+\alpha_n$ and
$\partial^\alpha=\partial^{\abs{\alpha}}/\bigl(\partial x_1^{\alpha_1}
\cdots \partial x_n^{\alpha_n}\bigr)$, $x_1,\ldots,x_n$ being the
generic independent variables in $\R^n$. In addition, given $l\in\N$
and $x=(x_1,\ldots,x_n)\in\R^n$, we write $\binom{l}{\alpha}=l! /
(\alpha_1!\cdots\alpha_n!)$ and $x^\alpha=x_1^{\alpha_1}\cdots
x_n^{\alpha_n}$. We shall make frequent use of the relation
\begin{equation}\label{eq:xToalpha}
   \abs{x}^{2l}=\sum_{\abs{\alpha}=l} \binom{l}{\alpha} x^{2\alpha},
\end{equation}
valid for any $l\in\N$ and $x\in\R^n$.

Let $\Omega$ be a nonempty open set in $\R^n$.  For any $r\in\N$
and for any $p\in[1,\infty)$, we shall denote by $W^{r,p}(\Omega)$
the usual Sobolev space defined by
\begin{equation*}
  W^{r,p}(\Omega)=\{\,v\in L^p(\Omega) \mid
  \forall\alpha\in\N^n,\ \abs{\alpha}\leq r,\ \partial^\alpha v\in
  L^p(\Omega) \,\}.
\end{equation*}
We recall that the derivatives $\partial^\alpha v$ are taken in the
distributional sense.  The space $W^{r,p}(\Omega)$ is
equipped with the semi-norms $\abs{\,\cdot\,}_{j,p,\Omega}$, with
$j\in\{0,\ldots,r\}$, and the norm $\norm{\,\cdot\,}_{r,p,\Omega}$
given by
\begin{equation*}
  \abs{v}_{j,p,\Omega}=\biggl( \sum_{\abs{\alpha}=j}
      \int_{\Omega} \abs{\partial^\alpha v(x)}^p\,dx \biggr)^{1/p}
  \text{\ and\ \ }
  \norm{v}_{r,p,\Omega}=\biggl( \sum_{j=0}^r
      \abs{v}_{j,p,\Omega}^p\biggr)^{1/p}.
\end{equation*}
For any $r\in(0,\infty)\setminus\N$ and for any $p\in[1,\infty)$,
we shall denote by $W^{r,p}(\Omega)$ the Sobolev space of
noninteger order $r$, formed by the (equivalence classes of)
functions $v\in W^{\floor{r},p}(\Omega)$ such that
\begin{equation*}
  \abs{v}_{r,p,\Omega}^p
  =\sum_{\abs{\alpha}=\floor{r}}\int_{\Omega\times\Omega}
  \frac{\abs{\partial^\alpha v(x)-\partial^\alpha
  v(y)}^p}{\abs{x-y}^{n+p(r-\floor{r})}}\,dx\,dy < \infty.
\end{equation*}
Besides the semi-norms
$\abs{\,\cdot\,}_{j,p,\Omega}$, with $j\in\{0,\ldots,\floor{r}\}$,
and $\abs{\,\cdot\,}_{r,p,\Omega}$, the space $W^{r,p}(\Omega)$ is
endowed with the norm
\begin{equation*}
  \norm{v}_{r,p,\Omega}=\Bigl(\norm{v}_{\floor{r},p,\Omega}^p
    + \abs{v}_{r,p,\Omega}^p\Bigr)^{1/p}.
\end{equation*}

Given $j\in\N$ and $v\in W^{j+1,p}(\Omega)$, we put
\begin{equation*}
   \abs{\nabla v}_{0,p,\Omega}=\biggl( \int_\Omega \abs{\nabla v(x)}^p
   \,dx \biggr)^{1/p} \quad\text{and}\quad
   \abs{\nabla v}_{j,p,\Omega}=\biggl( \sum_{\abs{\alpha}=j} 
   \abs{\nabla(\partial^\alpha v)}_{0,p,\Omega}^p
   \biggr)^{1/p}.
\end{equation*}
The mapping $v\mapsto\abs{\nabla v}_{j,p,\Omega}$ is a semi-norm in
$W^{j+1,p}(\Omega)$ equivalent to $\abs{\,\cdot\,}_{j+1,p,\Omega}$.

We shall use the following definition of the Fourier transform
$\hat{v}$ of a function $v\in L^1(\R^n)$:
\begin{equation*}
   \forall\xi\in\R^n,\ \hat{v}(\xi)=\int_{\R^n} 
   v(x) e^{-i x\cdot\xi}\,dx.
\end{equation*}
We refer to standard textbooks for the properties of the Fourier 
transform and their extension to the space $\mathcal{S}'(\R^n)$ of 
tempered distributions. We just recall the following result:
\begin{equation}\label{eq:pdFT}
   \forall v\in\mathcal{S}'(\R^n), \forall\alpha\in\N^n, 
   \ i^{\abs{\alpha}}\xi^\alpha \hat{v}
   =\widehat{\partial^\alpha v}.
\end{equation}

% ======================================================================
% ======================================================================
% 3. Results in the general case
% ======================================================================
% ======================================================================

\section{General results for $p\in[1,\infty)$}\label{s3}

As mentioned in the introduction, for a smooth bounded domain 
$\Omega$ or for $\Omega=\R^n$, we are interested in obtaining the 
following limit:
\begin{equation}\label{eq:intro0bis}
   \lim_{\sigma\to\ell} \sigma^{1-k}(1-\sigma)^{k}
   \abs{v}_{l+\sigma,p,\Omega}^p,
\end{equation}
with $\ell\in\{0^{+},1^{-}\}$, $k\in\{0,1\}$, $l\in\N$,
$p\in[1,\infty)$ and $v$ belonging to a suitable Sobolev space. For
$\Omega=\R^n$, we shall study the cases $(\ell,k)=(0^+,0)$ and
$(1^-,1)$, whereas, for $\Omega$ bounded, we shall consider the cases
$(\ell,k)=(0^+,1)$ and $(1^-,1)$, taking into account that
$\lim_{\sigma\to 0^+}(1-\sigma)=1$. The limit corresponding to any
other combination of $\ell$ and $k$ follows trivially from the above
cases.

% ----------------------------------------------------------------------
\begin{thm}\label{thm33}
   Let $\Omega$ be a bounded domain in $\R^n$ with a
   Lipschitz-conti\-nu\-ous boundary. Let $p\in[1,\infty)$ and $l\in\N$.
   Then, for any $v\in W^{l+1,p}(\Omega)$,
   \begin{equation}\label{eq:lims1Omega}
      \lim_{\sigma\to1^{-}}
      (1-\sigma)\abs{v}_{l+\sigma,p,\Omega}^p=
      p^{-1}K_{p,n}\abs{\nabla v}_{l,p,\Omega}^p,
   \end{equation}
   where
   \begin{equation}\label{eq:defKpn}
      K_{p,n}=\int_{S_{n-1}}\abs{\omega\cdot\nu}^p\,d\omega,
   \end{equation}
   $\nu$ being any unit vector in $\R^n$.
\end{thm}

\begin{proof}
   The case $l=0$ is a result by Bourgain, Br\'ezis and Mironescu
   (cf.~\cite{BouBreMir01}). For the sake of completeness, we just
   clarify here some details of their proof. We use, however, the
   notations in~\cite{Bre02}, which are slightly simpler. Let
   $(\rho_\varepsilon)_{\varepsilon>0}$ be any family of nonnegative
   functions, contained in $L^1_{\mathrm{loc}}(0,\infty)$, such that
   \begin{equation*}
      \int_0^\infty \rho_\varepsilon(t)t^{n-1}\,dt=1,
      \ \forall \varepsilon>0,
      \quad \text{and}\quad 
      \ \lim_{\varepsilon\to0}\int_\delta^\infty 
      \rho_\varepsilon(t)t^{n-1}\,dt =0,\ \forall \delta>0.
   \end{equation*}
   It follows from Theorems~2 and~3 in~\cite{BouBreMir01} that, for 
   any $v\in W^{1,p}(\Omega)$, 
   \begin{equation}\label{eq1:th31}
      \lim_{\varepsilon\to0}\int_{\Omega\times\Omega}
      \frac{\abs{v(x)-v(y)}^p}{\abs{x-y}^p}
      \rho_\varepsilon(\abs{x-y})\,dx\,dy=K_{p,n}\abs{\nabla 
      v}_{0,p,\Omega}^p,
   \end{equation}
   where $K_{p,n}$ is defined by \eqref{eq:defKpn}. Let us choose the
   family $(\rho_\varepsilon)_{\varepsilon>0}$ given by
   \begin{equation*}
      \rho_\varepsilon(t)=\begin{cases}
         \varepsilon d^{-\varepsilon} t^{\varepsilon-n}, &
	    \text{if } t\leq d, \\
         0,& \text{if } t> d,
      \end{cases}
   \end{equation*}
   $d$ being the diameter of $\Omega$. Then, \eqref{eq1:th31} becomes
   \begin{equation*}
      \lim_{\varepsilon\to0} \varepsilon d^{-\varepsilon}
      \int_{\Omega\times\Omega}
      \frac{\abs{v(x)-v(y)}^p}{\abs{x-y}^{n+p-\varepsilon}}
      \,dx\,dy=K_{p,n}\abs{\nabla v}_{0,p,\Omega}^p,
   \end{equation*}
   which implies \eqref{eq:lims1Omega}, for $l=0$, if we replace
   $\varepsilon$ by $p(1-\sigma)$.
   
   Let us now consider the case $l\geq 1$. Since the $l$th-order 
   derivatives of functions in $W^{l+1,p}(\Omega)$ belong to
   $W^{1,p}(\Omega)$, by the case $l=0$, for any
   $v\in W^{l+1,p}(\Omega)$, we have
   \begin{align*}
      \lim_{\sigma\to1^{-}}(1-\sigma)\abs{v}_{l+\sigma,p,\Omega}^p 
      &=\lim_{\sigma\to1^{-}} (1-\sigma) \sum_{\abs{\alpha}=l}
           \abs{\partial^\alpha v}_{\sigma,p,\Omega}^p \\
      &=\sum_{\abs{\alpha}=l} \lim_{\sigma\to1^{-}} 
           (1-\sigma)\abs{\partial^\alpha v}_{\sigma,p,\Omega}^p \\
      &=\sum_{\abs{\alpha}=l} p^{-1}K_{p,n}
      \abs{\nabla(\partial^\alpha v)}_{0,p,\Omega}^p
      =p^{-1}K_{p,n}\abs{\nabla v}_{l,p,\Omega}^p,
   \end{align*}
   which yields \eqref{eq:lims1Omega}.
\end{proof}
% ----------------------------------------------------------------------

% ----------------------------------------------------------------------
\begin{rem}\label{rem:valKpn}
   Let us provide the explicit value of the constant $K_{p,n}$ given
   by \eqref{eq:defKpn}. Since the definition of $K_{p,n}$ is
   independent of the unit vector $\nu$, we can take
   $\nu=(1,0,\ldots,0)$. On the one hand, we have
   \begin{align*}
      \int_{x_1^2+\cdots+x_n^2\leq 1}\abs{x_1}^p\,dx 
      &=\int_0^1\Bigl(\int_{S_{n-1}}t^{n-1}\abs{t\omega_1}^p
           \,d\omega\Bigr)dt \\
      &=\Bigl(\int_{S_{n-1}}\abs{\omega\cdot\nu}^p
           \,d\omega\Bigr)  \int_0^1t^{n-1+p}dt
      =\frac{K_{p,n}}{n+p}.
   \end{align*}
   On the other hand,
   \begin{multline*}
      \int_{x_1^2+\cdots+x_n^2\leq 1}\abs{x_1}^p\,dx
      =\int_{-1}^1\abs{x_1}^p\Bigl(
        \int_{x_2^2+\cdots+x_n^2\leq 1-x_1^2}\,dx_2\cdots 
        dx_n\Bigr)dx_1 \\
      =\vartheta_{n-1}\int_{-1}^1 \abs{x_1}^p(1-x_1^2)^{(n-1)/2}\,dx_1 
      =2\vartheta_{n-1}\int_0^1 x_1^p\,(1-x_1^2)^{(n-1)/2}\,dx_1 \\
      =\vartheta_{n-1}\int_0^1 t^{(p-1)/2}\,(1-t)^{(n-1)/2}\,dt 
      =\vartheta_{n-1}B\Bigl(\frac{p+1}{2},\frac{n+1}{2}\Bigr),
   \end{multline*}
   where $\vartheta_{n-1}$ is the volume of the unit ball in
   $\R^{n-1}$ and $B$ is the Euler Beta function. Hence,
   \begin{equation}\label{eq:valKpn}
      K_{p,n}=(n+p)\vartheta_{n-1}
      B\Bigl(\frac{p+1}{2},\frac{n+1}{2}\Bigr)
      =\frac{2\pi^{(n-1)/2}\Gamma((p+1)/2)}{\Gamma((n+p)/2)},
   \end{equation}
   where $\Gamma$ stands for the Euler Gamma function. Although
   Theorem~\ref{thm33} only requires the value of $K_{p,n}$ for
   $p\geq1$, the above expression is valid, in fact, for any $p\geq0$.
\end{rem}
% ----------------------------------------------------------------------

% ----------------------------------------------------------------------
\begin{thm}\label{thm31}
   Let $p\in[1,\infty)$ and $l\in\N$. Then, for any $v\in
   W^{l+1,p}(\R^n)$,
   \begin{equation}\label{eq:limThm31}
      \lim_{\sigma\to1^{-}}
      (1-\sigma)\abs{v}_{l+\sigma,p,\R^n}^p=
      p^{-1}K_{p,n}\abs{\nabla v}_{l,p,\R^n}^p,
   \end{equation}
   where $K_{p,n}$ is given by \eqref{eq:defKpn}.
\end{thm}

\begin{proof}
   This result, for $l=0$, is usually credited to Bourgain, Br\'ezis
   and Mironescu~\cite{BouBreMir01}, since it is implicitly contained
   in their paper. It can be proved from Theorem~\ref{thm33}, first
   for smooth functions with compact support and then, by density, for
   any element in $W^{l+1,p}(\R^n)$. An explicit proof is given by
   Milman~\cite[Subsection 3.1]{Mil05}, but without providing the
   precise definition of the constant $K_{p,n}$, which can be deduced
   from Karadzhov, Milman and Xiao~\cite[p.~332]{KarMilXia05}. The
   case $l>0$ is identical to that in the proof of
   Theorem~\ref{thm33}.
\end{proof}
% ----------------------------------------------------------------------

% ----------------------------------------------------------------------
\begin{thm}\label{thm32}
   Let $p\in[1,\infty)$, $l\in\N$ and $\sigma_0\in(0,1)$. Then, for
   any $v\in W^{l+\sigma_0,p}(\R^n)$,
   \begin{equation}\label{eq:limThm32}
      \lim_{\sigma\to0^{+}}
      \sigma\abs{v}_{l+\sigma,p,\R^n}^p=
      \frac{4\pi^{n/2}}{p\,\Gamma(n/2)}\abs{v}_{l,p,\R^n}^p.
   \end{equation}
\end{thm}

\begin{proof}
   Maz'ya and Shaposhnikova proved in \cite[Theorem~3]{MazSha02} that
   \eqref{eq:MS} holds for any $v$ belonging to
   $\bigcup_{0<\sigma<1}W^{\sigma,p}_0(\R^n)$, where
   $W^{\sigma,p}_0(\R^n)$ stands for the completion of
   $C^\infty_0(\R^n)$ with respect to
   $\abs{\,\cdot\,}_{\sigma,p,\R^n}$ (which is a norm in this last
   space). The condition on $v$ can be relaxed to
   $v\in\bigcup_{0<\sigma<\sigma_0}W^{\sigma,p}_0(\R^n)$ for some
   $\sigma_0\in(0,1)$. Likewise, since $C^\infty_0(\R^n)$ is dense in
   $W^{\sigma,p}(\R^n)$ with respect to
   $\norm{\,\cdot\,}_{\sigma,p,\R^n}=\bigl(\abs{\,\cdot\,}_{0,p,\R^n}^p
   +\abs{\,\cdot\,}_{\sigma,p,\R^n}^p\bigr)^{1/p}$, it follows that
   $W^{\sigma,p}(\R^n)\subset W^{\sigma,p}_0(\R^n)$. Thus, taking into
   account the embedding $W^{\sigma_0,p}(\R^n)\hookrightarrow
   W^{\sigma,p}(\R^n)$, if $\sigma_0\geq\sigma$, and that
   $\abs{S_{n-1}}=2\pi^{n/2}/\Gamma(n/2)$, we conclude that, for
   $l=0$, \eqref{eq:limThm32} follows from Maz'ya and Shaposhnikova's
   result.
   
   Now, let us assume that $l\geq 1$. Given $v\in 
   W^{l+\sigma_0,p}(\R^n)$, it is clear that any $l$th-derivative 
   $\partial^\alpha v$ belongs to $W^{\sigma_0,p}(\R^n)$. The 
   case $l=0$ implies that
   \begin{multline*}
      \lim_{\sigma\to0^{+}}\sigma\abs{v}_{l+\sigma,p,\R^n}^p
      =\lim_{\sigma\to0^{+}} \sigma \sum_{\abs{\alpha}=l}
      \abs{\partial^\alpha v}_{\sigma,p,\R^n}^p \\
      =\sum_{\abs{\alpha}=l} \lim_{\sigma\to0^{+}} 
      \sigma\abs{\partial^\alpha v}_{\sigma,p,\R^n}^p
      =\sum_{\abs{\alpha}=l} \dfrac{4\pi^{n/2}}{p\,\Gamma(n/2)}
      \abs{\partial^\alpha v}_{0,p,\R^n}^p
      =\dfrac{4\pi^{n/2}}{p\,\Gamma(n/2)}
      \abs{v}_{l,p,\R^n}^p.
   \end{multline*}
   The theorem follows.
\end{proof}
% ----------------------------------------------------------------------

As we shall next see, there exists a qualitative difference in the
behaviour of $\abs{v}_{l+\sigma,p,\Omega}$ as $\sigma\to0^{+}$
depending on whether $\Omega$ is $\R^n$ or a bounded set.
Theorem~\ref{thm32} implies that the semi-norm
$\abs{v}_{l+\sigma,p,\R^n}$ blows up to infinity (except for
polynomials of degree $\leq l$) as $\sigma\to0^{+}$. However, for a
bounded set $\Omega$, a priori, the semi-norm
$\abs{v}_{l+\sigma,p,\Omega}$ may remain bounded. In fact, this is
always the case. Even more, as $\sigma\to0^{+}$, that semi-norm tends
to Dini's semi-norm $\abs{v}_{l,\mathrm{Dini}(p),\Omega}$, defined, 
following Milman~\cite{Mil05}, by
\begin{equation*}
  \abs{v}_{l,\mathrm{Dini}(p),\Omega}^p
  =\sum_{\abs{\alpha}=l}\int_{\Omega\times\Omega}
  \frac{\abs{\partial^\alpha v(x)-\partial^\alpha
  v(y)}^p}{\abs{x-y}^n}\,dx\,dy.
\end{equation*}
Let us state and establish this result. We borrow the arguments from 
Milman~\cite[Theorem~3 and Example~2]{Mil05}.

% ----------------------------------------------------------------------
\begin{thm}\label{thm35}
   Let $\Omega$ be a bounded domain of $\R^n$ with a Lipschitz
   continuous boundary. Let $p\in[1,\infty)$, $l\in\N$ and
   $\sigma_0\in(0,1)$. Then, for any $v\in W^{l+\sigma_0,p}(\Omega)$,
   we have $\abs{v}_{l,\mathrm{Dini}(p),\Omega}<\infty$ and 
   \begin{equation*}
      \lim_{\sigma\to0^{+}}\abs{v}_{l+\sigma,p,\Omega}
      =\abs{v}_{l,\mathrm{Dini}(p),\Omega}.
   \end{equation*}
\end{thm}

\begin{proof}
   As in previous results, it suffices to prove the case $l=0$. Let
   $R$ be the diameter of $\Omega$. We consider the bijective linear
   mapping $F:\R^n\to\R^n$ given by $F(\hat{x})=R\,\hat{x}$ and we
   write $\widehat{\Omega}=F^{-1}(\Omega)$. Since $R=\diam\Omega$, it
   is clear that $\diam\widehat{\Omega}=1$. Thus,
   \begin{equation*}
      \forall \sigma\in(0,\sigma_0),\ \forall 
      \hat{x},\hat{y}\in\widehat{\Omega},
      \ 1\geq \abs{\hat{x}-\hat{y}}^{\sigma}
      \geq\abs{\hat{x}-\hat{y}}^{\sigma_0}.
   \end{equation*}
   Consequently, given $\hat{v}\in W^{\sigma_0,p}(\widehat{\Omega})$,
   we have
   \begin{equation*}
      \forall \sigma\in(0,\sigma_0),
      \ \abs{\hat{v}}_{0,\mathrm{Dini}(p),\widehat{\Omega}}^p
      \leq\abs{\hat{v}}_{\sigma,p,\widehat{\Omega}}^p
      \leq\abs{\hat{v}}_{\sigma_0,p,\widehat{\Omega}}^p<\infty.
   \end{equation*}
   Hence, by the Lebesgue's Dominated Convergence Theorem, we get
   \begin{align*}
      \lim_{\sigma\to0^{+}}\abs{\hat{v}}_{\sigma,p,\widehat{\Omega}}^p
      &=\lim_{\sigma\to0^{+}} \int_{\widehat{\Omega}\times\widehat{\Omega}}
       \frac{\abs{\hat{v}(\hat{x})-\hat{v}(\hat{y})}^p}%
         {\abs{\hat{x}-\hat{y}}^{n+p\sigma}}\,d\hat{x}\,d\hat{y} \\
      &=\int_{\widehat{\Omega}\times\widehat{\Omega}}\lim_{\sigma\to0^{+}}
       \frac{\abs{\hat{v}(\hat{x})-\hat{v}(\hat{y})}^p}%
         {\abs{\hat{x}-\hat{y}}^{n+p\sigma}}\,d\hat{x}\,d\hat{y} \\
       &=\int_{\widehat{\Omega}\times\widehat{\Omega}}
       \frac{\abs{\hat{v}(\hat{x})-\hat{v}(\hat{y})}^p}%
         {\abs{\hat{x}-\hat{y}}^n}\,d\hat{x}\,d\hat{y}
      =\abs{\hat{v}}_{0,\mathrm{Dini}(p),\widehat{\Omega}}^p.
   \end{align*}
   Now, for any $v\in W^{\sigma_0,p}(\Omega)$, the function 
   $\hat{v}=v\circ F$ belongs to $W^{\sigma_0,p}(\widehat{\Omega})$, 
   since 
   \begin{equation*}
     \abs{v}_{\sigma_0,p,\Omega}=R^{-\sigma_0+n/p}
     \abs{\hat{v}}_{\sigma_0,p,\widehat{\Omega}}.
   \end{equation*}
   Likewise, 
   \begin{equation*}
     \abs{v}_{0,\mathrm{Dini}(p),\Omega}=R^{n/p}
     \abs{\hat{v}}_{0,\mathrm{Dini}(p),\widehat{\Omega}}
   \end{equation*}
   and, for any $\sigma\in(0,\sigma_0)$, 
   \begin{equation*}
      \abs{v}_{\sigma,p,\Omega}=R^{-\sigma+n/p}
      \abs{\hat{v}}_{\sigma,p,\widehat{\Omega}}.
   \end{equation*}
   From these relations, we deduce that 
   $\abs{v}_{0,\mathrm{Dini}(p),\Omega}$ is finite and that
   \begin{equation*}
      \lim_{\sigma\to0^{+}}\abs{v}_{\sigma,p,\Omega}
      = \lim_{\sigma\to0^{+}} R^{-\sigma+n/p}
      \abs{\hat{v}}_{\sigma,p,\widehat{\Omega}}
      = R^{n/p}\abs{\hat{v}}_{0,\mathrm{Dini}(p),\widehat{\Omega}}
      =\abs{v}_{0,\mathrm{Dini}(p),\Omega}.
   \end{equation*}
   The proof is complete.
\end{proof}
% ----------------------------------------------------------------------

% ----------------------------------------------------------------------
\begin{rem}
   It is worth noting that, under the conditions of
   Theorem~\ref{thm35}, the arguments in its proof lead, in general,
   to the following bound:
   \begin{equation*}
      \forall v\in W^{l+\sigma_0,p}(\Omega),
      \ \abs{v}_{l,\mathrm{Dini}(p),\Omega}
      \leq R^\sigma \abs{v}_{l+\sigma,p,\Omega}
      \leq R^{\sigma_0} \abs{v}_{l+\sigma_0,p,\Omega},
   \end{equation*}
   with $R=\diam\Omega$.
\end{rem}
% ----------------------------------------------------------------------

% ----------------------------------------------------------------------
\begin{rem}\label{rem:Dini}
   By a change of variables and Fubini's Theorem, it can be seen that
   \begin{equation*}
      \abs{v}_{0,\mathrm{Dini}(p),\Omega}=\left(n\int_0^{+\infty}
      \frac{\overline{\omega}(v,t)_p^p}{t}\,dt\right)^{1/p},
   \end{equation*}
   where $\overline{\omega}(v,t)_p$ is the averaged modulus of 
   smoothness, given by
   \begin{equation*}
      \overline{\omega}(v,t)_p^p= t^{-n}\int_{\abs{h}\leq t}
      \abs{\Delta_h v}_{0,p,\Omega}^p\,dh,\qquad t>0,
   \end{equation*}
   with $\Delta_h v(x)=v(x+h)-v(x)$, if $x,x+h\in\Omega$, and
   $\Delta_hf(x)=0$, otherwise.
   Hence, for $l=0$, Theorem~\ref{thm35} establishes that, for any
   $v\in W^{\sigma_0,p}(\Omega)$, the function
   $\overline{\omega}(v,\,\cdot\,)_p$ satisfies a Dini-type
   condition. Analogous comments can be made for $l>0$. This justifies
   the name given to the semi-norm
   $\abs{\,\cdot\,}_{l,\mathrm{Dini}(p),\Omega}$.
   Likewise, since $\overline{\omega}(v,t)_p$ is equivalent to the 
   usual modulus of smoothness 
   $\omega(v,t)_p=\sup_{\abs{h}\leq t}\abs{\Delta_h v}_{0,p,\Omega}$,
   Theorem~\ref{thm35} comprises as a particular case the result given
   by Milman (cf.~\cite[Example~2]{Mil05}).
\end{rem}
% ----------------------------------------------------------------------

% ----------------------------------------------------------------------
\begin{rem}
   The semi-norm $\abs{\,\cdot\,}_{r,p,\R^n}$ can be normalized as 
   follows:
   \begin{equation}\label{eq:normalizRn}
      \babs{v}_{r,p,\R^n}=\lambda_{\sigma,p}\abs{v}_{r,p,\R^n},
   \end{equation}
   where $\sigma=r-\floor{r}$ and
   \begin{equation*}
      \lambda_{\sigma,p}=
      \begin{cases}
         \bigl(\sigma(1-\sigma)\bigr)^{1/p},
            & \text{if } \sigma\in(0,1),  \\
         1, & \text{if } \sigma=0.
      \end{cases}
   \end{equation*}
   Then, the semi-norm $\babs{\,\cdot\,}_{r,p,\R^n}$ is continuous in 
   the scale of Sobolev spaces $\bigl(W^{r,p}(\R^n)\bigr)_{r\geq0}$ 
   in the following sense:
   \begin{align*}
      \forall r>0,\ \forall v\in W^{r,p}(\R^n),
      \ \lim_{s\to r^{-}}\babs{v}_{s,p,\R^n}&\approx
      \babs{v}_{r,p,\R^n},  \\
      \forall r\geq0,\ \forall\epsilon>0,
      \ \forall v\in W^{r+\epsilon,p}(\R^n),
      \ \lim_{s\to r^{+}}\babs{v}_{s,p,\R^n}&\approx
      \babs{v}_{r,p,\R^n},
   \end{align*}
   where the symbol $\approx$ means that there exist two positive 
   constants $c_1$ and $c_2$, independent of $v$,  such that
   \begin{equation*}
      c_1 \babs{v}_{r,p,\R^n}\leq \lim_{s\to r^{\pm}}\babs{v}_{s,p,\R^n}
      \leq c_2 \babs{v}_{r,p,\R^n}.
   \end{equation*}
   In fact, if $r\notin\N$, both lateral limits are equal to
   $\babs{v}_{r,p,\R^n}$. For $r\in\N$, these relations are direct
   consequences of Theorems~\ref{thm31} and~\ref{thm32}, whereas, for
   $r\notin\N$, they come from the Lebesgue's Dominated Convergence
   Theorem.
   
   For a bounded domain $\Omega\subset\R^n$ with a Lipschitz
   continuous boundary, we could also consider the normalization
   $\babs{v}_{r,p,\Omega}=\lambda_{\sigma,p}\abs{v}_{r,p,\Omega}$.
   But, due to Theorem~\ref{thm35}, for any $r\in\N$, we would get
   \[
      \forall\epsilon>0,
      \ \forall v\in W^{r+\epsilon,p}(\Omega),
      \ \lim_{s\to r^{+}}\babs{v}_{s,p,\Omega}=0,
   \]
   which is quite unnatural. A better normalization is
   \begin{equation*}
      \cabs{v}_{r,p,\Omega}=(1-\sigma)^{1/p}\abs{v}_{r,p,\Omega},
   \end{equation*}
   with $\sigma=r-\floor{r}$.
   We now have:
   \begin{align*}
      \forall r>0,\ r\notin\N,\ \forall v\in W^{r,p}(\Omega),
      \ \lim_{s\to r^{-}}\cabs{v}_{s,p,\Omega}&\approx
      \cabs{v}_{r,p,\Omega}, \\
      \forall r\geq0,\ \forall\epsilon>0,
      \ \forall v\in W^{r+\epsilon,p}(\Omega),
      \ \lim_{s\to r^{+}}\cabs{v}_{s,p,\Omega}&\approx
      \begin{cases}
        \cabs{v}_{r,p,\Omega},& \text{if }r\notin\N, \\
	\abs{v}_{r,\mathrm{Dini}(p),\Omega},& \text{if }r\in\N.
      \end{cases}
   \end{align*}
   Observe that, given $r\in\N$ and $\varepsilon>0$, the semi-norms
   $\abs{\,\cdot\,}_{r,\mathrm{Dini}(p),\Omega}$ and
   $\abs{\,\cdot\,}_{r,p,\Omega}$ are not equivalent on
   $W^{r+\epsilon,p}(\Omega)$
   ($\abs{\,\cdot\,}_{r,\mathrm{Dini}(p),\Omega}$ is null for
   polynomials of degree $\leq r$, while
   $\abs{\,\cdot\,}_{r,p,\Omega}$ is null only for polynomials of
   degree $\leq r-1$). Consequently, the semi-norm
   $\cabs{\,\cdot\,}_{r,p,\Omega}$ is not right-continuous for
   $r\in\N$.
\end{rem}
% ----------------------------------------------------------------------

% ======================================================================
% ======================================================================
% 4. The particular case p=2
% ======================================================================
% ======================================================================

\section{The particular case $p=2$}\label{s4}

The purpose of this section is to provide an alternative proof of
Theorems~\ref{thm31} and~\ref{thm32} based on the Fourier transform. 
We start with several preliminary results.

% ----------------------------------------------------------------------
\begin{lem}\label{lem41}
   For any $n\in\N$ and $\sigma\in(0,1)$, there exists a positive
   constant $\Gsn$ such that
   \begin{equation}\label{eq2:lem41}
      \forall\xi\in\R^n,
      \ \int_{\R^n}\frac{\abs{e^{i\xi.y}-1}^2}{\abs{y}^{n+2\sigma}}
      \,dy=\Gsn\,\abs{\xi}^{2\sigma}.
   \end{equation}
\end{lem}

\begin{proof}
   The relation \eqref{eq2:lem41} is obviously true if $\xi=0$, so 
   let us assume that $\xi\ne0$. Let $\nu=\xi/\abs{\xi}$. We have
   \begin{align*}
      \int_{\R^n}&\frac{\abs{e^{i\xi\cdot y}-1}^2}{\abs{y}^{n+2\sigma}}
         \,dy \\
      &=\abs{\xi}^{2\sigma}
        \int_{\R^n}\frac{\abs{e^{i\nu\cdot x}-1}^2}{\abs{x}^{n+2\sigma}}\,dx
        && (\text{change $x=\abs{\xi}y$}) \\
     &=\abs{\xi}^{2\sigma}\int_{S_{n-1}} \left(\int_0^{+\infty}\rho^{n-1}
       \frac{\abs{e^{i\rho \nu\cdot\omega}-1}^2}%
       {\abs{\rho \omega}^{n+2\sigma}}\,d\rho
       \right)d\omega \\
     &=\abs{\xi}^{2\sigma}\int_{S_{n-1}} \left(\int_0^{+\infty}
       \frac{\abs{e^{i\rho \nu\cdot\omega}-1}^2}{\rho^{1+2\sigma}}\,d\rho
       \right)d\omega  \\
     &=\abs{\xi}^{2\sigma}\int_{S_{n-1}} \left(\int_0^{+\infty}
       \frac{2(1-\cos(\rho \nu\cdot\omega))}{\rho^{1+2\sigma}}\,d\rho
       \right)d\omega \\
     &=\abs{\xi}^{2\sigma}\int_{S_{n-1}} \left(\int_0^{+\infty}
       \frac{2(1-\cos(\rho \abs{\nu\cdot\omega}))}{\rho^{1+2\sigma}}\,d\rho
       \right)d\omega && \text{($\cos$ is even)} \\
     &=\abs{\xi}^{2\sigma}
       \int_{S_{n-1}} \abs{\nu\cdot\omega}^{2\sigma}\left(\int_0^{+\infty}
       \frac{2(1-\cos t)}{t^{1+2\sigma}}\,dt
       \right)d\omega && \text{(change $t=\rho\abs{\nu\cdot\omega}$)} \\
     &=\abs{\xi}^{2\sigma}K_{2\sigma,n}\,M_\sigma,
   \end{align*}
   where $K_{2\sigma,n}$ is given by \eqref{eq:defKpn} with
   $p=2\sigma$ and
   \begin{equation}\label{eq:Ms}
     M_\sigma=\int_0^{+\infty}
       \frac{2(1-\cos t)}{t^{1+2\sigma}}\,dt,
   \end{equation}
   which is convergent for any $\sigma\in(0,1)$. It then suffices
   to take $\Gsn=K_{2\sigma,n}\,M_\sigma$.
\end{proof}
% ----------------------------------------------------------------------

% ----------------------------------------------------------------------
\begin{rem}\label{rem41}
   For any $n\in\N$ and $\sigma\in(0,1)$, let us show that
   \begin{equation}\label{eq1:lem41}
      \Gsn=\frac{2\pi^{(n+1)/2}\Gamma(\sigma+1/2)}%
      {\Gamma(\sigma+n/2)\Gamma(1+2\sigma)\sin(\pi\sigma)}.
   \end{equation}
   Integrating by parts in \eqref{eq:Ms}, we get
   \begin{equation*}
     M_\sigma
     = -\frac{(1-\cos t)}{\sigma t^{2\sigma}}\bigg]^{+\infty}_0
     +\int_0^{+\infty}\frac{\sin t}{\sigma t^{2\sigma}}\,dt 
     =\frac{1}{\sigma}\int_0^{+\infty}\frac{\sin t}{t^{2\sigma}}\,dt.
   \end{equation*}
   This last integral can be computed in several ways. For example,
   the cases $\sigma\in(0,1/2)$, $\sigma=1$ and $\sigma\in(1/2,1)$ are
   covered, respectively, by relations 3.764.1, 3.741.2 and, after an
   integration by parts, 3.764.2 in Gradshteyn and
   Ryzhik~\cite{GraRyz07}. Using well-known properties of the Gamma
   function, as well as the identity
   \begin{equation*}
    \Gamma(z)\Gamma(1-z)=\frac{\pi}{\sin \pi z},\quad z\notin\Z,
   \end{equation*}
   we finally derive that
   \begin{equation}\label{eq3:lem41}
      M_\sigma=\frac{\pi}{\Gamma(1+2\sigma)\sin(\pi\sigma)},
   \end{equation}
   Since $\Gsn=K_{2\sigma,n}\,M_\sigma$, this relation,
   together with \eqref{eq:valKpn}, implies \eqref{eq1:lem41}.
\end{rem}
% ----------------------------------------------------------------------

% ----------------------------------------------------------------------
\begin{pro}[C. Goulaouic]\label{pro1}
   Let $\sigma\in(0,1)$. Then
   \begin{equation}\label{eq0:pro1}
      W^{\sigma,2}(\R^n)=L^2(\R^n)\cap\widetilde{H}^\sigma(\R^n),
   \end{equation}
   with
   \begin{equation}\label{eq:defH}
      \widetilde{H}^\sigma=\Bigl\{ v\in\mathcal{S}'(\R^n) \Bigm\vert  
      \hat{v}\in L^1_{\mathrm{loc}}(\R^n),\ 
      \int_{\R^n}\abs{\xi}^{2\sigma} 
      \abs{\hat{v}(\xi)}^2\,d\xi <\infty \Bigr\}.
   \end{equation}
   In fact, for any $v\in W^{\sigma,2}(\R^n)$,
   \begin{equation}\label{eq:pro1}
      \abs{v}_{\sigma,2,\R^n}^2=(2\pi)^{-n}\, \Gsn
      \int_{\R^n}\abs{\xi}^{2\sigma}\abs{\hat{v}(\xi)}^2\,d\xi,
   \end{equation}
   where $\Gsn$ is the constant given by Lemma~\ref{lem41}.
\end{pro}

\begin{proof}
   Let $v\in L^2(\R^n)$. We first remark that $v$ is, in particular, a
   tempered distribution and, by Plancherel's Theorem, $\hat{v}\in
   L^2(\R^n)$, so $\hat{v}\in L^1_{\mathrm{loc}}(\R^n)$. Thus, to
   prove \eqref{eq0:pro1}, it suffices to see that the semi-norm
   $\abs{v}_{\sigma,2,\R^n}$ is finite if and only if the integral
   $\int_{\R^n}\abs{\xi}^{2\sigma}\abs{\hat{v}(\xi)}^2\,d\xi$ is
   finite. But this is a consequence of \eqref{eq:pro1}. So let us
   show that \eqref{eq:pro1} holds. To this end, we follow the
   reasoning of Goulaouic~\cite[p.~101]{Gou81}.
   
   For any $y\in\R^n$, the Fourier transform of the translated
   function $x\mapsto v(x+y)$ is the function $\xi\mapsto
   e^{iy\cdot\xi}\,\hat{v}(\xi)$. Hence, by Parseval's identity, we
   have
   \begin{equation*}
      \int_{\R^n}\abs{v(x+y)-v(x)}^2\,dx=
      (2\pi)^{-n}\int_{\R^n}\abs{\hat{v}(\xi)}^2
      \abs{e^{iy\cdot\xi}-1}^2\,d\xi.
   \end{equation*}
   Then, by Fubini's Theorem and Lemma~\ref{lem41}, we finally deduce 
   that
  \begin{align*}
      \abs{v}_{\sigma,2,\R^n}^2
      &=\int_{\R^n\times\R^n}
         \frac{\abs{v(x+y)-v(x)}^2}{\abs{y}^{n+2\sigma}}\,dx\,dy \\
      &=(2\pi)^{-n}\int_{\R^n}\abs{\hat{v}(\xi)}^2
         \biggl(\int_{\R^n}
         \frac{\abs{e^{iy\cdot\xi}-1}^2}{\abs{y}^{n+2\sigma}} 
         \,dy\biggr)\,d\xi \\
      &=(2\pi)^{-n}\,\Gsn
         \int_{\R^n}\abs{\xi}^{2\sigma}\abs{\hat{v}(\xi)}^2\,d\xi,
  \end{align*}
  which yields \eqref{eq:pro1} and completes the proof.
\end{proof}
% ----------------------------------------------------------------------

% ----------------------------------------------------------------------
\begin{lem}\label{lem1a}
   Let $\sigma_0\in(0,1)$. Then, for any $v\in W^{\sigma_0,2}(\R^n)$,
   \begin{equation*}
      \lim_{\sigma\to0^{+}} \int_{\R^n}
      \abs{\xi}^{2\sigma}\abs{\hat{v}(\xi)}^2\,d\xi
      =(2\pi)^n\abs{v}_{0,2,\R^n}^2.
   \end{equation*}
\end{lem}

\begin{proof}
   Let $v\in W^{\sigma_0,2}(\R^n)$. For any $\sigma\in(0,\sigma_0]$,
   let us consider the integral
   $I_{\sigma}=\int_{\R^n}g_{\sigma}(\xi)\,d\xi$, where
   $g_{\sigma}(\xi)=(1-\abs{\xi}^{2\sigma})\abs{\hat{v}(\xi)}^2$. This
   integral is well defined: since $v\in W^{\sigma_0,2}(\R^n)$, $v$
   also belongs to $L^2(\R^n)$ and $W^{\sigma,2}(\R^n)$, so 
   $\hat{v}\in L^2(\R^n)$ and, by Proposition~\ref{pro1},
   $v\in\widetilde{H}^\sigma(\R^n)$.
   
   Let $r$ and $R$ be numbers such that $0<r\leq 1<R$. We set
   \begin{equation*}
      I_{\sigma}=\int_{\abs{\xi}\leq r} g_{\sigma}(\xi)\,d\xi
      + \int_{r<\abs{\xi}<R} g_{\sigma}(\xi)\,d\xi
      +\int_{\abs{\xi}\geq R}g_{\sigma}(\xi)\,d\xi=J_1+J_2+J_3.
   \end{equation*}
   Let $\varepsilon>0$ be given. Let us show that we can choose $r$,
   $R$ and $\sigma\in (0,\sigma_0)$ such that
   $\abs{I_{\sigma}}<\varepsilon$. We have
   \begin{equation*}
      \abs{J_1}\leq \int_{\abs{\xi}\leq r}\abs{\hat{v}(\xi)}^2\,d\xi.
   \end{equation*}
   Clearly, $\abs{J_1}\leq \varepsilon/3$ for $r$ small enough, since
   $\hat{v}\in L^2(\R^n)$. Moreover,
   \begin{equation*}
      \abs{J_3}\leq \int_{\abs{\xi}\geq R}\abs{\hat{v}(\xi)}^2\,d\xi 
      + \int_{\abs{\xi}\geq R}\abs{\xi}^{2\sigma_0}
      \abs{\hat{v}(\xi)}^2\,d\xi,
   \end{equation*}
   and the two terms on the right member are arbitrarily small when $R$ is
   large enough, the first, because $\hat{v}\in L^2(\R^n)$ and the
   second, because, by Proposition \ref{pro1},
   $v\in\widetilde{H}^{\sigma_0}(\R^n)$. So, $\abs{J_3}<\varepsilon/3$
   for $R$ sufficiently large. Once $r$ and $R$ chosen, it suffices to
   take $\sigma$ small enough to achieve $\abs{J_2}<\varepsilon/3$.
  
   The preceding reasoning implies that
   \begin{equation*}
      \lim_{\sigma\to0^{+}}\int_{\R^n}g_{\sigma}(\xi)\,d\xi=0.
   \end{equation*}
   Consequently, taking Plancherel's Theorem into account, we
   conclude that
   \begin{equation*}
      \lim_{\sigma\to0^{+}}\int_{\R^n}\abs{\xi}^{2\sigma} 
      \abs{\hat{v}(\xi)}^2\,d\xi=\int_{\R^n}\abs{\hat{v}(\xi)}^2\,d\xi
      =(2\pi)^n\abs{v}_{0,2,\R^n}^2.\qedhere
   \end{equation*}
\end{proof}
% ----------------------------------------------------------------------

% ----------------------------------------------------------------------
\begin{lem}\label{lem1b}
   For any $v\in W^{1,2}(\R^n)$,
   \begin{equation*}
      \lim_{\sigma\to1^{-}} \int_{\R^n} 
      \abs{\xi}^{2\sigma}\abs{\hat{v}(\xi)}^2\,d\xi
      =(2\pi)^n\abs{\nabla v}_{0,2,\R^n}^2.
   \end{equation*}
\end{lem}

\begin{proof}
   Let $v\in W^{1,2}(\R^n)$. For any $\sigma\in(0,1)$, we now
   consider the integral
   $I_{\sigma}=\int_{\R^n}g_{\sigma}(\xi)\,d\xi$, with
   $g_{\sigma}(\xi)=(\abs{\xi}^2-\abs{\xi}^{2\sigma})\abs{\hat{v}(\xi)}^2$.
   It is clear that $\abs{I_\sigma}<\infty$: on the one hand, the 
   embedding $W^{1,2}(\R^n)\hookrightarrow W^{\sigma,2}(\R^n)$ and 
   Proposition~\ref{pro1} imply that
   $v\in\widetilde{H}^\sigma(\R^n)$; on the 
   other hand, since $v\in W^{1,2}(\R^n)$,
   \begin{multline}\label{eq:lem1b}
      \int_{\R^n}\abs{\xi}^2\abs{\hat{v}(\xi)}^2\,d\xi
      =\sum_{\abs{\beta}=1}\int_{\R^n} \xi^{2\beta}
       \abs{\hat{v}(\xi)}^2\,d\xi \\
      =\sum_{\abs{\beta}=1}\int_{\R^n} 
       \abs{i\,\xi^{\beta}\hat{v}(\xi)}^2\,d\xi
      =\sum_{\abs{\beta}=1}\int_{\R^n} 
       \abs{\widehat{\partial^\beta v}(\xi)}^2\,d\xi \\
      =\sum_{\abs{\beta}=1} (2\pi)^n \int_{\R^n} 
       \abs{\partial^\beta v(x)}^2\,dx
      =(2\pi)^n \abs{\nabla v}_{0,2,\R^n}^2,
   \end{multline}
   which is finite.

   As in the proof of Lemma~\ref{lem1a}, we set
   \begin{equation*}
      I_{\sigma}=\int_{\abs{\xi}\leq r} g_{\sigma}(\xi)\,d\xi
      + \int_{r<\abs{\xi}<R} g_{\sigma}(\xi)\,d\xi
      +\int_{\abs{\xi}\geq R}g_{\sigma}(\xi)\,d\xi=J_1+J_2+J_3,
   \end{equation*}
   with $r$ and $R$ such that $0<r\leq 1<R$. Let $\varepsilon>0$ be given.
   Clearly, we have
   \begin{equation*}
      \abs{J_1}\leq 2\int_{\abs{\xi}\leq r}\abs{\hat{v}(\xi)}^2\,d\xi
      \qquad\text{and}\qquad
      \abs{J_3}\leq 2\int_{\abs{\xi}\geq R}\abs{\xi}^2
      \abs{\hat{v}(\xi)}^2\,d\xi.
  \end{equation*}
  Then, the assumption $v\in W^{1,2}(\R^n)$ implies that $r$ and $R$ can
  be chosen in such a way that $\abs{J_1}$ and $\abs{J_3}$ be
  $\leq\varepsilon/3$. We have just to take $\sigma$ sufficiently
  close to $1$ to achieve $\abs{J_2}<\varepsilon/3$. Consequently,
  \begin{equation*}
      \lim_{\sigma\to1^{+}}\int_{\R^n}g_{\sigma}(\xi)\,d\xi=0.
   \end{equation*}
   From this relation and \eqref{eq:lem1b}, we finally derive that
   \begin{equation*}
      \lim_{\sigma\to1^{-}}\int_{\R^n}\abs{\xi}^{2\sigma} 
      \abs{\hat{v}(\xi)}^2\,d\xi
      =\int_{\R^n}\abs{\xi}^2\abs{\hat{v}(\xi)}^2\,d\xi
      =(2\pi)^n\abs{\nabla v}_{0,2,\R^n}^2.\qedhere
   \end{equation*} 
\end{proof}
% ----------------------------------------------------------------------

We are now ready to prove the main result in this section, which
establishes Theorems~\ref{thm31} and \ref{thm32} in the particular case
$p=2$. The reader may want to check that the constants on the
right-hand side of \eqref{eq:limThm31} and \eqref{eq:limThm32} are
equal, for $p=2$, to those in \eqref{eq:lim1Thm41} and 
\eqref{eq:lim0Thm41}, respectively.

% ----------------------------------------------------------------------
\begin{thm}\label{thm1}
   Let $l\in\N$.
   \begin{enumerate}[(i)]
      \item Let $\sigma_0\in(0,1)$. Then, for any $v\in 
      W^{l+\sigma_0,2}(\R^n)$,
      \begin{equation}\label{eq:lim0Thm41}
         \lim_{\sigma\to0^{+}}\sigma\abs{v}_{l+\sigma,2,\R^n}^2
         =\dfrac{2\pi^{n/2}}{\Gamma(n/2)}
         \abs{v}_{l,2,\R^n}^2.
      \end{equation}
      \item For any $v\in W^{l+1,2}(\R^n)$,
      \begin{equation}\label{eq:lim1Thm41}
         \lim_{\sigma\to1^{-}}(1-\sigma)\abs{v}_{l+\sigma,2,\R^n}^2
         =\dfrac{\pi^{n/2}}{n\Gamma(n/2)}
         \abs{\nabla v}_{l,2,\R^n}^2.
      \end{equation}
   \end{enumerate}
\end{thm}

\begin{proof}
   Let us first assume that $l=0$. It readily follows from
   \eqref{eq1:lem41} and the continuity and properties of the $\Gamma$
   function that
   \begin{equation*}
      \lim_{\sigma\to0^{+}}
      \sigma\,\Gsn=\dfrac{2\pi^{n/2}}{\Gamma(n/2)} 
      \qquad\text{and}\qquad
      \lim_{\sigma\to1^{-}}
      (1-\sigma) \Gsn=\dfrac{\pi^{n/2}}{n\Gamma(n/2)}.
   \end{equation*}
   Consequently, by Proposition~\ref{pro1} and Lemma~\ref{lem1a}, we 
   have
   \begin{equation*}
      \lim_{\sigma\to0^{+}} \sigma\abs{v}_{\sigma,2,\R^n}^2
      =\lim_{\sigma\to0^{+}} \sigma (2\pi)^{-n} \Gsn
      \int_{\R^n}\abs{\xi}^{2\sigma}\abs{\hat{v}(\xi)}^2 \,d\xi
      =\dfrac{2\pi^{n/2}}{\Gamma(n/2)}
      \abs{v}_{0,2,\R^n}^2.
   \end{equation*}
   Likewise, by Proposition~\ref{pro1} and Lemma~\ref{lem1b},
   \begin{multline*}
      \lim_{\sigma\to1^{-}} (1-\sigma)\abs{v}_{\sigma,2,\R^n}^2 \\
      =\lim_{\sigma\to1^{-}} (1-\sigma) (2\pi)^{-n} \Gsn
      \int_{\R^n}\abs{\xi}^{2\sigma}\abs{\hat{v}(\xi)}^2 \,d\xi 
      =\dfrac{\pi^{n/2}}{n\Gamma(n/2)}
      \abs{\nabla v}_{0,2,\R^n}^2.
   \end{multline*}
   The reasoning for $l\geq 1$ follows the same pattern already shown 
   in Theorems~\ref{thm31} and \ref{thm32}.
\end{proof}
% ----------------------------------------------------------------------

In the proof of Theorem~\ref{thm1} and the preceding lemmas,
Proposition~\ref{pro1} plays a fundamental role. This result can be
extended to characterize the space $W^{r,2}(\R^n)$ for any $r\geq 0$.
Although it is not required here, we include such an extension in this
section for the sake of completeness.

% ----------------------------------------------------------------------
\begin{thm}\label{thm:last}
   Let $r\in[0,\infty)$. Then
   \begin{equation}\label{eq:WisLintH}
      W^{r,2}(\R^n)=L^2(\R^n)\cap\widetilde{H}^r(\R^n),
   \end{equation}
   where $\widetilde{H}^r(\R^n)$ is given by \eqref{eq:defH} with $r$ 
   instead of $\sigma$.
   Moreover, for any $m\in\N$ and $s\geq0$ such that $r=m+s$,
   \begin{equation}\label{eq:WisLintX}
      W^{r,2}(\R^n)=L^2(\R^n)\cap X^{m,s},
   \end{equation}
   where $X^{m,s}=\bigl\{ v\in\mathcal{D}'(\R^n) \bigm\vert
   \forall\alpha\in\N^n,\ \abs{\alpha}=m, \ \partial^\alpha
   v\in\widetilde{H}^s(\R^n) \}$,
   $\mathcal{D}'(\R^n)$ being the space of distributions on $\R^n$.
\end{thm}

\begin{proof}
   We put $r=l+\sigma$, with $l=\floor{r}$ and $\sigma\in[0,1)$. Let
   $m\in\N$ and $s\geq0$ such that $r=m+s$. We remark that $m\leq l$.
   
   Since $L^2(\R^n)\subset
   \mathcal{S}'(\R^n)$ and $L^2(\R^n)\subset
   L^1_{\mathrm{loc}}(\R^n)$, it is clear that
   \begin{equation}\label{eq:LintH}
      L^2(\R^n)\cap\widetilde{H}^r(\R^n)= \Bigl\{ v\in L^2(\R^n)\Bigm\vert
      \int_{\R^n}\abs{\xi}^{2r} \abs{\hat{v}(\xi)}^2\,d\xi <\infty \Bigr\}
   \end{equation}
   and
   \begin{multline}\label{eq:LintX}
      L^2(\R^n)\cap X^{m,s}= \Bigl\{ v\in L^2(\R^n)\Bigm\vert
      \forall\alpha\in\N^n,\ \abs{\alpha}=m, \\
      \ \widehat{\partial^\alpha v}\in L^1_{\mathrm{loc}}(\R^n)
      \text{ and }\int_{\R^n}\abs{\xi}^{2s} 
      \abs{\widehat{\partial^\alpha v}(\xi)}^2\,d\xi <\infty \Bigr\}.
   \end{multline}
   We divide the proof into several steps: Steps~1 and~2 prove 
   \eqref{eq:WisLintH}, whereas Steps~3 and~4 establish 
   \eqref{eq:WisLintX}.
   
   \Step{1} $W^{r,2}(\R^n)\subset L^2(\R^n)\cap\widetilde{H}^r(\R^n)$. \\
   Let $v\in W^{r,2}(\R^n)$. By \eqref{eq:LintH}, we have
   just to show that
   $\int_{\R^n}\abs{\xi}^{2r}\abs{\hat{v}(\xi)}^2\,d\xi$ is finite.
   Let us first consider the case $\sigma\in(0,1)$. Every $l$-th
   derivative $\partial^\alpha v$ belongs to $W^{\sigma,2}(\R^n)$. By
   Proposition~\ref{pro1}, we have
   \begin{align*}
      &\int_{\R^n}\abs{\xi}^{2r}\abs{\hat{v}(\xi)}^2\,d\xi \\
      &\qquad=\int_{\R^n}\abs{\xi}^{2\sigma}\abs{\xi}^{2l}
            \abs{\hat{v}(\xi)}^2\,d\xi 
      =\sum_{\abs{\alpha}=l}\binom{l}{\alpha}
            \int_{\R^n}\abs{\xi}^{2\sigma}\xi^{2\alpha}
	    \abs{\hat{v}(\xi)}^2\,d\xi \\
      &\qquad=\sum_{\abs{\alpha}=l}\binom{l}{\alpha}
        \int_{\R^n}\abs{\xi}^{2\sigma}
        \abs{\widehat{\partial^\alpha v}(\xi)}^2\,d\xi 
      =\sum_{\abs{\alpha}=l}\binom{l}{\alpha}
        (2\pi)^{n}\, \Gsn^{-1} 
        \abs{\partial^\alpha v}_{\sigma,2,\R^n}^2 \\
      &\qquad\leq M (2\pi)^{n}\, \Gsn^{-1} \sum_{\abs{\alpha}=l}
        \abs{\partial^\alpha v}_{\sigma,2,\R^n}^2 
      = M (2\pi)^{n}\, \Gsn^{-1} \abs{v}_{r,2,\R^n}^2 <\infty,
   \end{align*}
   with $M=\max\bigl\{ \binom{l}{\alpha} \bigm\vert \alpha\in\N^n,
   \ \abs{\alpha}=l \bigr\}$. If $\sigma=0$, the above reasoning is
   still valid, taking $\Gsn=1$ and using Plancherel's Theorem instead
   of Proposition~\ref{pro1}.
   
   \Step{2} $L^2(\R^n)\cap\widetilde{H}^r(\R^n)\subset W^{r,2}(\R^n)$.\\
   Let $v\in L^2(\R^n)\cap\widetilde{H}^r(\R^n)$. For any 
   $\alpha\in\N^n$ such that $\abs{\alpha}\leq l$, we have
   \begin{align*}
      \int_{\R^n} \abs{\xi}^{2r} \abs{\hat{v}(\xi)}^2\,d\xi
      &=\int_{\R^n}\abs{\xi}^{2(r-\abs{\alpha})} 
      \abs{\xi}^{2\abs{\alpha}} \abs{\hat{v}(\xi)}^2\,d\xi \\
      &=\sum_{\abs{\beta}=\abs{\alpha}} \binom{\abs{\alpha}}{\beta}
      \int_{\R^n}\abs{\xi}^{2(r-\abs{\alpha})} 
      \xi^{2\beta} \abs{\hat{v}(\xi)}^2\,d\xi.
   \end{align*}
   Consequently, $\widehat{\partial^\alpha v}=i^{\abs{\alpha}}\xi^\alpha 
   \hat{v}$ belongs to $L^2(\R^n)$,  since
   \begin{align*}
      \int_{\R^n} \abs{\xi^\alpha\hat{v}(\xi)}^2\,d\xi
      &=\int_{\abs{\xi}<1} \xi^{2\alpha}\abs{\hat{v}(\xi)}^2\,d\xi
         +\int_{\abs{\xi}\geq1} \xi^{2\alpha}\abs{\hat{v}(\xi)}^2\,d\xi \\
      &\leq \int_{\abs{\xi}<1} \abs{\hat{v}(\xi)}^2\,d\xi
         + \binom{\abs{\alpha}}{\alpha} \int_{\abs{\xi}\geq 1} 
         \abs{\xi}^{2(r-\abs{\alpha})} \xi^{2\alpha} 
	 \abs{\hat{v}(\xi)}^2\,d\xi \\
      &\leq \int_{\R^n} \abs{\hat{v}(\xi)}^2\,d\xi
         +\int_{\R^n} \abs{\xi}^{2r} \abs{\hat{v}(\xi)}^2\,d\xi <\infty.
   \end{align*}
   We deduce from Plancherel's Theorem that $v\in W^{l,2}(\R^n)$. If
   $\sigma\in(0,1)$, we still have to see that $\abs{v}_{r,2,\R^n}$ is
   finite. But a reasoning analogous to that in Step~1 shows, as 
   desired, that
   \begin{equation*}
      \abs{v}_{r,2,\R^n}^2\leq (2\pi)^{-n}\, \Gsn \int_{\R^n} 
        \abs{\xi}^{2r}\abs{\hat{v}(\xi)}^2\,d\xi <\infty.
   \end{equation*}
   
   \Step{3} $L^2(\R^n)\cap X^{m,s}\subset W^{r,2}(\R^n)$.\\
   Let $v\in L^2(\R^n)\cap X^{m,s}$. Then, taking 
   \eqref{eq:LintX} into account, we have
   \begin{align*}
      \int_{\R^n} \abs{\xi}^{2r} \abs{\hat{v}(\xi)}^2\,d\xi
      &=\int_{\R^n}\abs{\xi}^{2s} 
      \abs{\xi}^{2m} \abs{\hat{v}(\xi)}^2\,d\xi \\
      &=\sum_{\abs{\alpha}=m} \binom{m}{\alpha}
      \int_{\R^n}\abs{\xi}^{2s} 
      \xi^{2\alpha} \abs{\hat{v}(\xi)}^2\,d\xi \\
      &=\sum_{\abs{\alpha}=m} \binom{m}{\alpha}
      \int_{\R^n}\abs{\xi}^{2s} 
      \abs{\widehat{\partial^\alpha v}(\xi)}^2\,d\xi 
      < \infty.
   \end{align*}
   Thus, it follows from \eqref{eq:WisLintH} and \eqref{eq:LintH} that
   $v\in W^{r,2}(\R^n)$.
   
   \Step{4} $W^{r,2}(\R^n)\subset L^2(\R^n)\cap X^{m,s}$.\\
   Let $v\in W^{r,2}(\R^n)$. Using \eqref{eq:WisLintH}, the
   reasoning in Step~2 shows that, for any $\alpha\in\N^n$ such that 
   $\abs{\alpha}=m$, $\widehat{\partial^\alpha v}$ belongs to 
   $L^2(\R^n)\subset L^1_{\mathrm{loc}}(\R^n)$ and
   \begin{align*}
      \int_{\R^n}\abs{\xi}^{2s} 
      \abs{\widehat{\partial^\alpha v}(\xi)}^2\,d\xi
      &=\int_{\R^n}\abs{\xi}^{2s} 
      \xi^{2\alpha} \abs{\hat{v}(\xi)}^2\,d\xi \\
      &\leq\sum_{\abs{\beta}=m} \binom{m}{\beta}
      \int_{\R^n}\abs{\xi}^{2s} 
      \xi^{2\beta} \abs{\hat{v}(\xi)}^2\,d\xi \\
      &=\int_{\R^n} \abs{\xi}^{2r} \abs{\hat{v}(\xi)}^2\,d\xi
      <\infty.
   \end{align*}
   We conclude that, by \eqref{eq:LintX}, $v\in L^2(\R^n)\cap X^{m,s}$.
\end{proof}
% ----------------------------------------------------------------------

% ----------------------------------------------------------------------
\begin{rem}
   Let $r>0$. Theorem~\ref{thm:last} allows us to endow $W^{r,2}(\R^n)$
   with semi-norms defined in $\widetilde{H}^r(\R^n)$ or $X^{m,s}$. 
   For example, the mapping
   \begin{equation*}
      \abs{\,\cdot\,}_{0,r}:v\mapsto\biggl(
       \int_{\R^n} \abs{\xi}^{2r} \abs{\hat{v}(\xi)}^2\,d\xi
       \biggr)^{1/2}
   \end{equation*}
   is a semi-norm in $\widetilde{H}^r(\R^n)$ (in fact, a hilbertian
   norm if $r<n/2$; cf. \cite{ArcLopTor04}), so it is in
   $W^{r,2}(\R^n)$. It follows from steps~1 and~2 in the proof of
   Theorem~\ref{thm:last} that $\abs{\,\cdot\,}_{0,r}$ and
   $\abs{\,\cdot\,}_{r,2,\R^n}$ are equivalent semi-norms. The
   equivalence constants depend on $\sigma$, since they contain
   $\Gsn$. In fact, taking into account \eqref{eq1:lem41} and the
   continuity of the Gamma function, it is readily seen that, given
   $l\in\N$, there exist constants $C_1$ and $C_2$, depending on $n$ 
   and $l$, such that, for all $\sigma\in(0,1)$ and
   $v\in W^{l+\sigma,2}(\R^n)$,
   \begin{equation*}
      C_1\abs{v}_{0,l+\sigma} \leq (2\sigma(1-\sigma))^{1/2}
      \abs{v}_{l+\sigma,2,\R^n}\leq C_2\abs{v}_{0,l+\sigma}.\qed
   \end{equation*}
   \let\qed\relax
\end{rem}
% ----------------------------------------------------------------------

% ======================================================================
% ======================================================================
% Acknowledgements
% ======================================================================
% ======================================================================

\subsection*{Acknowledgements}
This work has been  supported by the Mi\-nis\-terio de Ciencia e
Innovaci\'on (Spain), through the Research Project MTM\-2009-07315.

% ======================================================================
% ======================================================================
% Bibliography
% ======================================================================
% ======================================================================

% \let\thebibliographyOLD\thebibliography
% \let\endthebibliographyOLD\endthebibliography
% \renewenvironment{thebibliography}[1]%
%    {\begin{thebibliographyOLD}\normalfont\normalsize\baselineskip=17pt}%
%    {\end{thebibliographyOLD}}
% \bibliographystyle{plain}
% \bibliography{BiblioArcTor2011}

\begin{thebibliography}{10}

\bibitem{ArcLopTor04}
R.~Arcang\'eli, M.~C. L\'opez~{d}e Silanes, and J.~J. Torrens.
\newblock {\em Multidimensional Minimizing Splines. Theory and Applications}.
\newblock Grenoble Science. Kluwer Academic Publishers, Boston, 2004.

\bibitem{ArcLopTor07}
R.~Arcang{\'e}li, M.~C. L\'opez~{d}e Silanes, and J.~J. Torrens.
\newblock An extension of a bound for functions in {S}obolev spaces, with
  applications to $(m,s)$-spline interpolation and smoothing.
\newblock {\em Numer. Math.}, 107(2):181--211, 2007.

\bibitem{ArcLopTor09}
R.~Arcang{\'e}li, M.~C. L\'opez~{d}e Silanes, and J.~J. Torrens.
\newblock Estimates for functions in {S}obolev spaces defined on unbounded
  domains.
\newblock {\em J. Approx. Theory}, 161:198--212, 2009.

\bibitem{ArcLopTor10}
R.~Arcang{\'e}li, M.~C. L\'opez~{d}e Silanes, and J.~J. Torrens.
\newblock Extensions of sampling inequalities to {S}obolev semi-norms of
  fractional order and derivative data.
\newblock {\em Numer. Math.}, (to appear).

\bibitem{BouBreMir01}
J.~Bourgain, H.~Br\'ezis, and P.~Mironescu.
\newblock Another look at {S}obolev spaces.
\newblock In J.~L. Menaldi, E.~Rofman, and A.~Sulem, editors, {\em Optimal
  Control and Partial Differential Equations. Volumen in honour of Professor
  Alain Bensoussan's 60th Birthday}, pages 439--455, Amsterdam, 2001. {IOS}
  Press.

\bibitem{Bre02}
H.~Br\'ezis.
\newblock How to recognize constant functions. {A} connection with {S}obolev
  spaces.
\newblock {\em Uspekhi Mat. Nauk}, 57(4(346)):59--74, 2002.

\bibitem{Gou81}
C.~Goulaouic.
\newblock {\em Analyse Fonctionnelle et Calcul Diff\'erentiel}.
\newblock \'Ecole Polytechnique, Paris, 1981.

\bibitem{GraRyz07}
I.~S. Gradshteyn and I.~M. Ryzhik.
\newblock {\em Table of integrals, series, and products}.
\newblock Elsevier/Academic Press, Amsterdam, seventh edition, 2007.
\newblock Translated from the Russian, Translation edited and with a preface by
  Alan Jeffrey and Daniel Zwillinger.

\bibitem{KarMilXia05}
G.~E. Karadzhov, M.~Milman, and J.~Xiao.
\newblock Limits of higher-order {B}esov spaces and sharp reiteration theorems.
\newblock {\em J. Funct. Anal.}, 221(2):323--339, 2005.

\bibitem{KolLer05}
V.~I. Kolyada and A.~K. Lerner.
\newblock On limiting embeddings of {B}esov spaces.
\newblock {\em Studia Math.}, 171(1):1--13, 2005.

\bibitem{MasNag78}
W.~Masja (V.~Maz'ya) and J.~Nagel.
\newblock \"{U}ber \"aquivalente {N}ormierung der anisotropen
  {F}unktionalr\"aume {$H^{\mu }(\mathbb{R}^n)$}.
\newblock {\em Beitr\"age Anal.}, 12:7--17, 1978.

\bibitem{MazSha02}
V.~Maz'ya and T.~Shaposhnikova.
\newblock On the {B}ourgain, {B}rezis, and {M}ironescu theorem concerning
  limiting embeddings of fractional {S}obolev spaces.
\newblock {\em J. Funct. Anal.}, 195(2):230--238, 2002.

\bibitem{Mil05}
M.~Milman.
\newblock Notes on limits of {S}obolev spaces and the continuity of
  interpolation scales.
\newblock {\em Trans. Amer. Math. Soc.}, 357(9):3425--3442, 2005.

\bibitem{Tri11}
H.~Triebel.
\newblock Limits of {B}esov norms.
\newblock {\em Arch. Math.}, 96(2):169--175, 2011.

\end{thebibliography}

\end{document}